\documentclass[12pt]{article}
\usepackage[]{amsmath,amssymb}
\usepackage{amscd}
\usepackage{latexsym}
\usepackage{cite}
\usepackage{amsthm}

\newtheorem{definition}{Definition}[section]
\newtheorem{theorem}[definition]{Theorem}
\newtheorem{lemma}[definition]{Lemma}
\newtheorem{corollary}[definition]{Corollary}

\newtheorem{example}[definition]{Example}

\newtheorem{proposition}[definition]{Proposition}

\begin{document}
\title{\bf Using Catalan words and a $q$-shuffle \\algebra to
describe 
the Beck PBW\\ basis for the positive part
of  $U_q(\widehat{\mathfrak{sl}}_2)$
}
 \author{
Paul Terwilliger 
}
\date{}

\maketitle
\begin{abstract} We consider the positive part $U^+_q$ of the quantized  enveloping algebra
$U_q(\widehat{\mathfrak{sl}}_2)$. The algebra $U^+_q$ has a presentation involving two generators and two relations, called the $q$-Serre relations.
There is a PBW basis for $U^+_q$ due to Damiani, and a PBW basis for $U^+_q$ due to Beck.
In 2019 we used Catalan words and a $q$-shuffle algebra to express 
the Damiani PBW basis in closed form. In this paper we use a similar approach to express the Beck PBW basis in closed form. We also consider how the Damiani PBW basis and the Beck PBW basis
are  related to the alternating PBW basis for $U^+_q$.
\bigskip


 \end{abstract}

 \section{Introduction}
 The quantized enveloping algebra $U_q(\widehat{\mathfrak{sl}}_2)$ appears  in representation theory \cite{charp}, statistical mechanics  \cite{hongkang, JM}, combinatorics \cite{TD00, drg, compactUqp},
 and the theory of tridiagonal pairs \cite{uqsl2hat, shape, nonnil,TwoRel, ItoTer}. In the present paper we consider a well known subalgebra $U^+_q$ of  $U_q(\widehat{\mathfrak{sl}}_2)$,
called the positive part \cite{baspp,
beck,bcp,
lusztig, pospart, xxz}.
 The algebra $U^+_q$  has a presentation involving
two generators $A$, $B$ and two relations, called the $q$-Serre relations:
 \begin{align*}
\lbrack A, \lbrack A, \lbrack A, B\rbrack_q \rbrack_{q^{-1}} \rbrack =0, \qquad \qquad 
\lbrack B, \lbrack B, \lbrack B, A \rbrack_q \rbrack_{q^{-1}}\rbrack = 0.
\end{align*}
In \cite{damiani} I. Damiani obtained a PBW basis for 
 $U^+_q$,
consisting of
some elements $\lbrace E_{n\delta+\alpha_0}\rbrace_{n=0}^\infty$,
$\lbrace E_{n\delta+\alpha_1}\rbrace_{n=0}^\infty$,
$\lbrace E_{n\delta}\rbrace_{n=1}^\infty$ that are
 defined recursively. In \cite{catalan} we expressed these elements in closed form, using Catalan words and
 a $q$-shuffle algebra. In Section 6 we will review this result in detail, and for now give a brief summary.
Start with a free associative algebra $\mathbb V$ on two generators $x,y$. These generators are called letters.
For an integer $n\geq 0$, a word of length $n$ in $\mathbb V$ is a product of letters $u_1 u_2 \cdots u_n$.
The vector space 
$\mathbb V$ has a basis consisting of its words; this basis is called standard.
In \cite{rosso1, rosso} M. Rosso introduced 
an associative algebra structure on $\mathbb V$, called a
$q$-shuffle algebra.
For letters $u,v$ their 
$q$-shuffle product is
$u\star v = uv+q^{( u,v) }vu$, where
$( u,v)=2$ 
(resp. $(u,v) =-2$) 
if $u=v$ (resp. 
 $u\not=v$).
 In \cite[Theorem~15]{rosso} Rosso gave an injective algebra homomorphism
$\natural$ from $U^+_q$ into the $q$-shuffle algebra
${\mathbb V}$, that sends $A\mapsto x$ and $B\mapsto y$.
In \cite{catalan} we applied $\natural $ to the Damiani PBW basis,
and expressed the image in the standard basis for $\mathbb V$.
This image involves words of the following type.
Define $\overline x = 1$
and $\overline y = -1$.
A word $u_1u_2\cdots u_n$ in $\mathbb V$ 
is said to be Catalan whenever
$\overline u_1+
\overline u_2+\cdots  + 
\overline u_i$ is nonnegative for 
$1 \leq i \leq n-1$ and zero for $i=n$.
In this case $n$ is even.
For $n\geq 0$ define
\begin{align*}
C_n =
  \sum u_1u_2\cdots u_{2n}
\lbrack 1\rbrack_q
\lbrack 1+\overline u_1\rbrack_q
\lbrack 1+\overline u_1+\overline u_2\rbrack_q
\cdots
\lbrack 1+\overline u_1+\overline u_2+ \cdots +\overline u_{2n}\rbrack_q,
\end{align*}
where the sum is over all the Catalan words $u_1 u_2 \cdots u_{2n}$
in $\mathbb V$ that have length $2n$. In \cite[Theorem~1.7]{catalan} we showed that the map
$\natural$ sends
\begin{align*}
E_{n\delta+\alpha_0} \mapsto  q^{-2n}(q-q^{-1})^{2n} xC_n, \qquad \quad 
E_{n\delta+\alpha_1} \mapsto
 q^{-2n}(q-q^{-1})^{2n} C_ny
 \end{align*}
 for $n\geq 0$, and
 \begin{align*}
 E_{n\delta} \mapsto  -q^{-2n}(q-q^{-1})^{2n-1} C_n
 \end{align*}
 for $n\geq 1$. In \cite[Proposition~6.1]{beck} J. Beck obtained a  PBW basis for $U^+_q$ by adjusting the Damiani PBW basis as follows.
 The elements $\lbrace E_{n\delta}\rbrace_{n=1}^\infty$ are replaced by some elements
 $\lbrace E^{\rm Beck}_{n \delta} \rbrace_{n=1}^\infty$ that satisfy the generating function identity below, see \cite[p.~6]{bcp}. 
 Referring
 to the exponential function and an indeterminate $t$,
 \begin{align*}
{\rm exp}  \Biggl((q-q^{-1}) \sum_{k=1}^\infty E^{\rm Beck}_{k\delta}  t^k \Biggr) = 1 - (q-q^{-1}) \sum_{k=1}^\infty E_{k\delta}  t^k.
\end{align*}
The main result of the present paper is that $\natural$ sends
  \begin{align*}
E^{\rm Beck}_{n\delta} \mapsto \frac{\lbrack 2n \rbrack_q}{n} q^{-2n} (q-q^{-1})^{2n-1} x C_{n-1} y
 \end{align*}
 for $n\geq 1$. In the above line, the notation $x C_{n-1}y$ refers to the free product.
 \medskip
 
 \noindent
We use our main result to obtain a number of corollaries and subsidiary results. For instance, we show that the following holds in the $q$-shuffle algebra $\mathbb V$:
\begin{align*}
&{\rm exp}  \Biggl( \sum_{k=1}^\infty \frac{ \lbrack 2 k \rbrack_q }{k} x C_{k-1} y t^k \Biggr) = 1 + \sum_{k=1}^\infty C_k t^k.   
\end{align*}
\noindent In the above line, the exponential function is with respect to the $q$-shuffle product, and the notation $xC_{k-1}y$ refers to the free product.
\medskip

\noindent 
In \cite{alternating} we introduced the alternating words in $\mathbb V$, and used them to obtain the alternating PBW basis for $U^+_q$ \cite[Theorem~10.1]{alternating}.
The following words are alternating:
\begin{align*}
&\tilde G_{1} = xy, \qquad \tilde G_{2} =
xyxy,\qquad \tilde G_3 = xyxyxy, \qquad \ldots
\end{align*}
Using  our main result and \cite[Proposition~11.8]{alternating}, we show that the following holds in the $q$-shuffle algebra $\mathbb V$:
\begin{align*}
&{\rm exp}  \Biggl(-\sum_{k=1}^\infty \frac{ (-1)^k \lbrack  k \rbrack_q }{k} x C_{k-1} y t^k \Biggr) = 1+ \sum_{k=1}^\infty \tilde G_k t^k.
\end{align*}
\noindent In the above line, the exponential function is with respect to the $q$-shuffle product, and the notation $xC_{k-1}y$ refers to the free product.
 \medskip

 \noindent The paper is organized as follows. Section 2 contains some preliminaries.
 In Section 3 we recall the algebra $U^+_q$. In Section 4, we review the PBW bases for $U^+_q$ due to Damiani and Beck.
 In Sections 5, 6 we review the embedding of $U^+_q$ into the $q$-shuffle algebra $\mathbb V$.
 Sections 7, 8 contain our main result and some corollaries.
 In Section 9, we apply our main result to the alternating words in $\mathbb V$.
 In Appendix A we give some examples that illustrate certain results from the main body of the paper.

 \section{Preliminaries} We now begin our formal argument. 
 Recall the natural numbers $\mathbb N = \lbrace 0,1,2,\ldots \rbrace$.
 Let $\mathbb F$ denote a field with characteristic zero. Throughout this paper, every  vector space we discuss is  over $\mathbb F$.
 Every algebra we discuss is  associative, over $\mathbb F$, and has a multiplicative identity. A subalgebra has the same multiplicative identity
 as the parent algebra.
  \begin{definition}\label{def:pbw}
 \rm 
(See \cite[p.~299]{damiani}.)
Let $ \mathcal A$ denote an algebra. A {\it Poincar\'e-Birkhoff-Witt} (or {\it PBW}) basis for $\mathcal A$
consists of a subset $\Omega \subseteq \mathcal A$ and a linear order $<$ on $\Omega$
such that the following is a basis for the vector space $\mathcal A$:
\begin{align*}
a_1 a_2 \cdots a_n \qquad n \in \mathbb N, \qquad a_1, a_2, \ldots, a_n \in \Omega, \qquad
a_1 \leq a_2 \leq \cdots \leq a_n.
\end{align*}
We interpret the empty product as the multiplicative identity in $\mathcal A$.
\end{definition}
 
 \noindent We will be discussing generating functions. Let $\mathcal A$ denote an algebra and let $t$ denote an indeterminate.
For a sequence $\lbrace a_k \rbrace_{k \in \mathbb N}$ of elements in $\mathcal A$, the corresponding generating function is
\begin{align*}
 a(t) = \sum_{k \in \mathbb N} a_k t^k.
 \end{align*} 
 The above sum is formal; issues of convergence are not considered.
 We call $a(t)$ the {\it generating function over $\mathcal A$ with coefficients $\lbrace a_k \rbrace_{k \in \mathbb N}$}. The coefficient $a_0$ is called the {\it constant coefficient}.
  For generating functions
 $a(t)=\sum_{k \in \mathbb N} a_k t^k$ and
 $b(t) = \sum_{k \in \mathbb N} b_k t^k$ over $\mathcal A$, their product $a(t)b(t)$ is the generating function $\sum_{k \in \mathbb N}c_k t^k$  such that 
 $c_k= \sum_{i=0}^k a_i b_{k-i}$ for $k\in \mathbb N$.
 The set of generating functions over $\mathcal A$ forms an algebra.
 Consider a generating function $a(t) = \sum_{k=1}^\infty a_k t^k$ over $\mathcal A$ with constant coefficient 0. Then the exponential 
 \begin{align*}
 {\rm exp}\, a(t) = \sum_{n\in \mathbb N} \frac{\bigl(a(t)\bigr)^n}{n!}
 \end{align*}
 is a generating function over $\mathcal A$ with constant coefficient 1. Moreover
 the natural logarithm 
 \begin{align*}
 {\rm ln}\bigl(1+ a(t)\bigr) = \sum_{n\in \mathbb N} \frac{(-1)^n \bigl(a(t)\bigr)^{n+1}}{n+1}
 \end{align*}
 is a generating function over $\mathcal A$ with constant coefficient 0. 
 Let  $a(t)=\sum_{k=1}^\infty a_k t^k$ and $b(t)=\sum_{k=1}^\infty b_kt^k$ denote generating functions over $\mathcal A$ that have constant coefficient 0. Then ${\rm exp}\, a(t) = 1+ b(t)$ if and only if  $a(t) = {\rm ln} \bigl(1+ b(t)\bigr)$.

 \begin{definition}\label{def:gr}\rm 
 A  {\it grading} of an algebra $\mathcal A$ is a sequence  $\lbrace \mathcal A_n \rbrace_{n \in \mathbb N}$ of subspaces of $\mathcal A$ such that
(i) $1 \in \mathcal A_0$; (ii) the sum $\mathcal A = \sum_{n \in \mathbb N} \mathcal A_n$ is direct; (iii) $\mathcal A_r \mathcal A_s \subseteq \mathcal A_{r+s} $ for $r,s\in \mathbb N$.
For $n \in \mathbb N$ the subspace $\mathcal A_n$ is called the {\it $n$-homogeneous component} of the grading.
\end{definition}

  \noindent Throughout the paper, fix a nonzero $q \in \mathbb F$
that is not a root of unity.
Recall the notation
\begin{align*}
\lbrack n\rbrack_q = \frac{q^n-q^{-n}}{q-q^{-1}}
\qquad \qquad n \in \mathbb N.
\end{align*}

\section{The algebra $U^+_q$}
In this section we recall the algebra $ U^+_q$.
\medskip

\noindent For elements $X, Y$ in any algebra, define their
commutator and $q$-commutator by 
\begin{align*}
\lbrack X, Y \rbrack = XY-YX, \qquad \qquad
\lbrack X, Y \rbrack_q = q XY- q^{-1}YX.
\end{align*}
\noindent Note that 
\begin{align}
\label{eq:qs}
\lbrack X, \lbrack X, \lbrack X, Y\rbrack_q \rbrack_{q^{-1}} \rbrack
= 
X^3Y-\lbrack 3\rbrack_q X^2YX+ 
\lbrack 3\rbrack_q XYX^2 -YX^3.
\end{align}

\begin{definition} \label{def:U} \rm
(See \cite[Corollary~3.2.6]{lusztig}.) 
Define the algebra $U^+_q$ by generators $A$, $B$ and relations
\begin{align}
\label{eq:qOns1}
&\lbrack A, \lbrack A, \lbrack A, B\rbrack_q \rbrack_{q^{-1}} \rbrack =0,
\\
\label{eq:qOns2}
&\lbrack B, \lbrack B, \lbrack B, A\rbrack_q \rbrack_{q^{-1}}\rbrack = 0.
\end{align}
We call $U^+_q$ the {\it positive part of $U_q(\widehat{\mathfrak{sl}}_2)$}.
The relations \eqref{eq:qOns1}, \eqref{eq:qOns2}  are called the {\it $q$-Serre relations}.
\end{definition}

\noindent For the moment abbreviate $U^+= U^+_q$. Since the $q$-Serre relations are homogeneous, the algebra $U^+$ has a grading
$\lbrace U^+_n \rbrace_{n \in \mathbb N}$ with the following property: for $n \in \mathbb N$ the subspace $U^+_n$ is spanned by the products $g_1 g_2 \cdots g_n$
such that $g_i$ is among $A$, $B$ for $1 \leq i \leq n$. In particular $U^+_0 = \mathbb F 1$ and $U^+_1$ is spanned by $A, B$.

\section{Two PBW bases for $U^+_q$}

\noindent In \cite{damiani},  Damiani obtained a PBW basis for $U^+_q$ that involves some elements
\begin{align}
\lbrace E_{n \delta+ \alpha_0} \rbrace_{n=0}^\infty,
\qquad \quad 
\lbrace E_{n \delta+ \alpha_1} \rbrace_{n=0}^\infty,
\qquad \quad 
\lbrace E_{n \delta} \rbrace_{n=1}^\infty.
\label{eq:Upbw}
\end{align}
These elements are recursively defined  as follows.  
\begin{align}
E_{\alpha_0} = A, \qquad \qquad
E_{\alpha_1} = B, \qquad \qquad
E_{\delta} = q^{-2}BA-AB
\label{eq:BAalt}
\end{align}
and for $n\geq 1$,
\begin{align}
&
E_{n \delta+\alpha_0} =
\frac{
\lbrack E_\delta, E_{(n-1)\delta+ \alpha_0} \rbrack
}
{q+q^{-1}},
\qquad \qquad
E_{n \delta+\alpha_1} =
 \frac{
 \lbrack
 E_{(n-1)\delta+ \alpha_1},
 E_\delta
 \rbrack
 }
 {q+q^{-1}},
 \label{eq:dam1introalt}
 \\
 &
 \qquad \qquad
 E_{n \delta} =
 q^{-2}  E_{(n-1)\delta+\alpha_1} A
 - A E_{(n-1)\delta+\alpha_1}.
 \label{eq:dam2introalt}
\end{align}


\begin{proposition}
    \label{prop:PBWbasis}
    {\rm (See \cite[p.~308]{damiani}.)}
A PBW basis for $U^+_q$ is obtained by the elements
{\rm \eqref{eq:Upbw}}
  in the 
     linear
    order
 \begin{align*}
 E_{\alpha_0} < E_{\delta+\alpha_0} <
  E_{2\delta+\alpha_0}
  < \cdots
  <
   E_{\delta} < E_{2\delta}
    < E_{3\delta}
 < \cdots
  <
   E_{2\delta + \alpha_1} <
     E_{\delta + \alpha_1} < E_{\alpha_1}.
     \end{align*}
   \end{proposition}
\noindent The PBW basis in   Proposition  \ref{prop:PBWbasis} will be called the {\it Damiani PBW basis}.

\begin{lemma}\label{lem:Ugr} For the grading $\lbrace U^+_n \rbrace_{n \in \mathbb N}$ of $U^+=U^+_q$, 
we have
$E_{n\delta+\alpha_0}, E_{n\delta+\alpha_1}\in U^+_{2n+1}$ for $n \geq 0$ and
$E_{n \delta}\in U^+_{2n}$
for $n \geq 1$.
\end{lemma}
\begin{proof} Use \eqref{eq:BAalt}--\eqref{eq:dam2introalt}.
 \end{proof}

\noindent 
 Next we recall some relations satisfied by the elements of the Damiani PBW basis.
\medskip

\begin{lemma}\label{lem:mc} {\rm (See \cite[p.~307]{damiani}.)} The elements $\lbrace E_{n\delta}\rbrace_{n=1}^\infty$ mutually
commute.
\end{lemma}
\begin{lemma}
\label{lem:dam2}
{\rm (See \cite[p.~307]{damiani}.)}
For $i,j \in \mathbb N$,
\begin{align*}
\lbrack E_{i\delta+\alpha_0}, E_{j\delta+\alpha_1}\rbrack_q = -q E_{(i+j+1)\delta}.
\end{align*}
\end{lemma}

\noindent We just gave some relations involving the elements of the Damiani PBW basis. Additional relations involving these elements can be found in \cite{damiani}, see also
\cite[Section~3]{catalan}.
\medskip

\noindent 
We have been discussing the Damiani PBW basis.
Next we discuss a variation on this PBW basis, due to J. Beck \cite{beck}.
Our discussion will involve some generating functions in an indeterminate $t$.

\begin{definition}  \rm {\rm (See \cite[p.~6]{bcp}.)}
Define the elements $\lbrace E^{\rm Beck}_{n \delta} \rbrace_{n=1}^\infty$ in $U^+_q$ such that
\begin{align}
\label{eq:expEE}
&{\rm exp}  \Biggl((q-q^{-1}) \sum_{k=1}^\infty E^{\rm Beck}_{k\delta}  t^k \Biggr) = 1 - (q-q^{-1}) \sum_{k=1}^\infty E_{k\delta}  t^k.
\end{align}
\end{definition}

\begin{example}\label{ex:EE} \rm We have
\begin{align*}
&E_\delta = - E^{\rm Beck}_\delta, \qquad \quad 
E_{2\delta} = - E^{\rm Beck}_{2\delta} -\frac{q-q^{-1}}{2} \bigl(E^{\rm Beck}_{\delta}\bigr)^2, 
\\
&E_{3\delta} = -E^{\rm Beck}_{3\delta}
-(q-q^{-1}) E^{\rm Beck}_{\delta} E^{\rm Beck}_{2\delta}
-\frac{(q-q^{-1})^2}{6} \bigl(E^{\rm Beck}_{\delta}\bigr)^3.
\end{align*}
Moreover
\begin{align*}
&E^{\rm Beck}_\delta = - E_\delta, \qquad \quad 
E^{\rm Beck}_{2\delta} = - E_{2\delta} -\frac{q-q^{-1}}{2} E^2_{\delta}, 
\\
&E^{\rm Beck}_{3\delta} = -E_{3\delta}
-(q-q^{-1}) E_{\delta} E_{2\delta}
-\frac{(q-q^{-1})^2}{3} E_{\delta}^3.
\end{align*}
\end{example}
\noindent We clarify how the elements $\lbrace E_{n\delta} \rbrace_{n=1}^\infty$ and $\lbrace E^{\rm Beck}_{n\delta} \rbrace_{n=1}^\infty$ are related.
\begin{lemma} \label{lem:hom} The following hold for $n\geq 1$:
\begin{enumerate}
\item[\rm (i)]  $E_{n\delta} $ is a homogeneous polynomial in $E^{\rm Beck}_\delta, E^{\rm Beck}_{2\delta}, \ldots, E^{\rm Beck}_{n\delta}$ that has total degree $n$, where we view
$E^{\rm Beck}_{k\delta}$ as having degree $k$ for $1 \leq k \leq n$;
\item[\rm (ii)] $E^{\rm Beck}_{n\delta}$ is a homogeneous polynomial in $E_\delta, E_{2\delta}, \ldots, E_{n\delta}$ that has total degree $n$, where we view
$E_{k\delta}$ as having degree $k$ for $1 \leq k \leq n$.
\end{enumerate}
\end{lemma}
\begin{proof} Use \eqref{eq:expEE} and induction on $n$.
\end{proof}
\begin{lemma}\label{lem:BeckG} For the grading $\lbrace U^+_n \rbrace_{n \in \mathbb N}$ of $U^+=U^+_q$, 
we have
$E^{\rm Beck}_{n \delta}\in U^+_{2n}$
for $n \geq 1$.
\end{lemma}
\begin{proof} By  Lemma \ref{lem:Ugr} and Lemma \ref{lem:hom}(ii).
\end{proof}

\begin{proposition} \label{prop:Beckpbw} 
{\rm (See \cite[Proposition~6.1]{beck}.)}
A PBW basis for $U^+_q$
is obtained by the elements 
 \begin{align*}
    \lbrace E_{n \delta + \alpha_0}\rbrace_{n=0}^\infty,
     \qquad \quad
      \lbrace E_{n \delta + \alpha_1}\rbrace_{n=0}^\infty,
       \qquad \quad
        \lbrace E^{\rm Beck}_{n \delta}\rbrace_{n=1}^\infty
         \end{align*}
   in the
     linear
    order
\begin{align*}
 &E_{\alpha_0} < E_{\delta+\alpha_0} <
  E_{2\delta+\alpha_0}
  <  \cdots
  <
   E^{\rm Beck}_{\delta} < E^{\rm Beck}_{2\delta}
    < E^{\rm Beck}_{3\delta}
 < \cdots 
<
   E_{2\delta + \alpha_1} <
     E_{\delta + \alpha_1} < E_{\alpha_1}.
     \end{align*}
\end{proposition}
\noindent The PBW basis in Proposition \ref{prop:Beckpbw} will be called the {\it Beck PBW basis}.
\medskip
\noindent
 Next we recall some relations satisfied by the elements of the Beck PBW basis.
\begin{lemma}\label{lem:Bmc} {\rm (See \cite[Proposition~1.2]{bcp}.)} The elements $\lbrace E^{\rm Beck}_{n\delta}\rbrace_{n=1}^\infty$ mutually
commute.
\end{lemma}

\begin{lemma} {\rm (See \cite[Proposition~1.2]{bcp}.)} For $k\geq 1$ and $\ell \geq 0$,
\begin{align}
\label{eq:con2}
\lbrack E_{\ell \delta+\alpha_0}, E^{\rm Beck}_{k \delta} \rbrack &= \frac{\lbrack 2k \rbrack_q}{k} E_{(k+\ell)\delta+\alpha_0},
\\
\label{eq:con1}
\lbrack E^{\rm Beck}_{k \delta}, E_{\ell \delta+\alpha_1} \rbrack &= \frac{\lbrack 2k \rbrack_q}{k} E_{(k+\ell)\delta+\alpha_1}.
\end{align}
\end{lemma}

\noindent In Section 6 we will return our attention to the Damiani PBW basis and the Beck PBW basis. In the meantime, we will discuss an embedding, due to Rosso\cite{rosso1, rosso},
of the algebra $U^+_q$ into a $q$-shuffle algebra. For this $q$-shuffle algebra, the underlying vector space is a free algebra on two generators. We denote this free algebra by $\mathbb V$.

 \section{The free algebra $\mathbb V$}
 \noindent
 Let $x$, $y$ denote noncommuting indeterminates.
 Let $\mathbb V$ denote the free algebra generated by $x$ and $y$.
 By a {\it letter} in $\mathbb V$ we mean $x$ or $y$. For $n \in \mathbb N$, by a {\it word of length $n$} in $\mathbb V$ we mean a product of letters
 $a_1 a_2\cdots a_n$. We interpret the word of length $0$ to be the multiplicative identity in $\mathbb V$; this word is called {\it trivial} and denoted by 1.
 The vector space $\mathbb V$ has a (linear) basis consisting of its words; this basis is called {\it standard}.
 We endow the vector space $\mathbb V$ with a  bilinear form $\langle\,,\,\rangle: \mathbb V \times \mathbb V \to \mathbb F$ with respect to which
 the standard basis is orthonormal. This bilinear form is symmetric and nondegenerate.
 For a subspace $W \subseteq \mathbb V$, recall its orthogonal complement $W^\perp = \lbrace v \in \mathbb V | \langle v,w\rangle = 0 \;\forall w \in W\rbrace$. 
\medskip
 
\noindent 
 For $n \in \mathbb N$ let $\mathbb V_n$ denote the subspace of $\mathbb V$ spanned by the words of length $n$. The sum $\mathbb V=\sum_{n\in \mathbb N} \mathbb V_n$ is direct and the summands
 are mutually orthogonal. We have $\mathbb V_0 = \mathbb F 1$. We have  $\mathbb V_r \mathbb V_s \subseteq \mathbb V_{r+s}$ for $r,s\in \mathbb N$. By these comments the sequence
 $\lbrace \mathbb V_n \rbrace_{n \in \mathbb N}$ is a grading of the algebra $\mathbb V$.
 \begin{lemma} \label{lem:XY}
 For $r,s \in \mathbb N$ and $X, X' \in \mathbb V_r$ and $Y, Y' \in \mathbb V_s$ we have
 \begin{align*}
 \langle XY, X'Y'\rangle = \langle X, X' \rangle \langle Y, Y'\rangle.
 \end{align*}
 \end{lemma}
 \begin{proof} Since the standard basis for $\mathbb V$ is orthonormal.
 \end{proof}
 
 \begin{definition}\label{def:Jpm} \rm Define $J^+, J^- \in \mathbb V$ by
 \begin{align*}
 J^+ &= xxxy-\lbrack 3 \rbrack_q xxyx + \lbrack 3 \rbrack_q xyxx- yxxx,
 \\
 J^- &= yyyx-\lbrack 3 \rbrack_q yyxy + \lbrack 3 \rbrack_q yxyy-xyyy.
 \end{align*}
 \end{definition}
 
 \begin{definition}\label{def:J}\rm Let $J$ denote the $2$-sided ideal of the free algebra $\mathbb V$ generated by $J^+, J^-$.
 \end{definition}
 \noindent  Consider the quotient algebra $\mathbb V/J$. Since the algebra $\mathbb V$ is freely generated by $x$ and $y$,
 there exists an algebra homomorphism $\xi:\mathbb V\to U^+_q$ that sends $x\mapsto A$ and $y \mapsto B$. The kernel of $\xi$
 is equal to $J$, in view of \eqref{eq:qs} and
 Definition \ref{def:U}. Therefore, $\xi$ induces an algebra isomorphism $\mathbb V/J \to U^+_q$ that sends $x+J \mapsto A$ and $y+J\mapsto B$.
 \medskip
 
 \noindent We just described a connection between $J$ and $U^+_q$. In Proposition \ref{prop:orthog} 
 we will describe another connection between $J$ and $U^+_q$.

 \section{The $q$-shuffle algebra $\mathbb V$}
 
 \noindent In the previous section we discussed the free algebra $\mathbb V$. There is
another algebra structure on $\mathbb V$,
called the $q$-shuffle algebra.
This algebra was introduced by Rosso
\cite{rosso1, rosso} and described further by Green
\cite{green}. We will adopt the approach of 
\cite{green}, which is suited to our purpose.
The $q$-shuffle product 
is denoted by $\star$. To describe this product, we start
with some special cases.
We have $1 \star v = v \star 1 = v$ for $v \in \mathbb V$.
 For
letters $u,v$ we have
\begin{align*}
u \star v = uv + vu q^{(u,v)}
\end{align*}
\noindent where
\bigskip
\centerline{
\begin{tabular}[t]{c|cc}
$(\,,\,)$ & $x$ & $y$
   \\  \hline
   $x$ &
   $2$ & $-2$
     \\
     $y$ &
      $-2$ & $2$
         \\
              \end{tabular}
              }
              \medskip

\noindent
 Thus
   \begin{align*}
    &x \star y = xy+ q^{-2}yx,
 \qquad \qquad \quad
   y \star x = yx + q^{-2}xy,
\\
&x \star x = (1+ q^2)xx \qquad \qquad \quad
    y\star y = (1+q^2)yy.
\end{align*}
\noindent For a letter $u$ and
  a nontrivial word $v= v_1v_2\cdots v_n$ in $\mathbb V$,
  \begin{align*}
  &u \star v =
  \sum_{i=0}^n v_1 \cdots v_{i} u v_{i+1} \cdots v_n
  q^{
  ( v_1, u)+
  ( v_2, u)+
  \cdots + ( v_{i}, u)},
  \\
  &v \star u = \sum_{i=0}^n v_1 \cdots v_{i} u v_{i+1} \cdots v_n
  q^{
  ( v_{n},u)
  +
  ( v_{n-1},u)
  +
  \cdots
  +
  (v_{i+1},u)
  }.
  \end{align*}
   For example
  \begin{align*}
     & x\star (yyy)= xyyy+ q^{-2} yxyy+ q^{-4} yyxy+q^{-6}yyyx,
    \\
    &(yyy)\star x = q^{-6}xyyy+ q^{-4} yxyy+ q^{-2} yyxy+yyyx.
      \end{align*}
\noindent For nontrivial words $u=u_1u_2\cdots u_r$
and $v=v_1v_2\cdots v_s$ in $\mathbb V$,
\begin{align}
\label{eq:uvcirc}
&u \star v  = u_1\bigl((u_2\cdots u_r) \star v\bigr)
+ v_1\bigl(u \star (v_2 \cdots v_s)\bigr)
q^{
( u_1, v_1) +
( u_2, v_1) +
\cdots
+
(u_r, v_1)},
\\
\label{eq:uvcirc2}
&u\star v =
\bigl(u \star (v_1 \cdots v_{s-1})\bigr)v_s +
\bigl((u_1 \cdots u_{r-1}) \star v\bigr)u_r
q^{
( u_r, v_1) +
(u_r, v_2) + \cdots +
(u_r, v_s)
}.
\end{align}
 For example, take $r=2$ and $s=2$. 
We have
\begin{align*}
 u \star v &= 
               u_1 u_2 v_1 v_2 
            \\
            &+
               u_1 v_1 u_2 v_2  q^{( u_2, v_1)}
             \\
	     &+ 
      u_1 v_1 v_2 u_2  q^{( u_2, v_1)+ ( u_2,v_2)}
             \\
	     &+ 
      v_1 u_1 u_2 v_2  q^{(u_1, v_1)+ ( u_2,v_1)}
             \\
	     &+ 
      v_1 u_1 v_2 u_2  q^{( u_1, v_1)+ 
       (u_2,v_1)+( u_2,v_2)}
             \\
	     &+ 
      v_1 v_2 u_1 u_2  q^{( u_1, v_1)+ 
       ( u_1,v_2)+(u_2,v_1)+ ( u_2,v_2)}.
\end{align*}
\noindent Above Lemma  \ref{lem:XY}
 we mentioned a grading of the free algebra $\mathbb V$. This is also a grading for the $q$-shuffle algebra $\mathbb V$.

\begin{definition}\label{def:Usub} \rm Let $U$ denote the subalgebra of the $q$-shuffle algebra $\mathbb V$ generated by $x,y$. 
\end{definition}
 \noindent 
The algebra $U$ is described as follows.
With some effort
(or by
\cite[Theorem~13]{rosso1},
\cite[p.~10]{green})
one
obtains
\begin{align}
&
x \star x \star x \star y -
\lbrack 3 \rbrack_q
x \star x\star y \star x +
\lbrack 3 \rbrack_q
x \star y \star x \star x -
y \star x \star x \star x  = 0,
\label{eq:qsc1}
\\
&
y \star y \star y \star x -
\lbrack 3 \rbrack_q
y \star y \star x \star y +
\lbrack 3 \rbrack_q
y \star x \star y \star y -
x \star y \star y \star y  = 0.
\label{eq:qsc2}
\end{align}
So in the $q$-shuffle algebra $\mathbb V$ the elements
$x,y$ satisfy the
$q$-Serre relations.
Consequently there exists an algebra homomorphism
$\natural$ from $U^+_q$ to the $q$-shuffle algebra $\mathbb V$,
that sends $A\mapsto x$ and $B\mapsto y$.
The map $\natural$ has image $U$
by Definition
\ref{def:Usub}, 
and is injective by
  \cite[Theorem~15]{rosso}. 
Therefore $\natural: U^+_q \to U$ is an algebra isomorphism. See \cite{grosse, leclerc,negut, boxq} for more information about the $q$-shuffle algebra $\mathbb V$ and its relationship to $U^+_q$.
\medskip

\noindent Earlier we mentioned a grading for both $U^+_q$ and the $q$-shuffle algebra $\mathbb V$. These gradings are related as follows.
The algebra $U$ inherits the grading of $U^+_q$ via $\natural$. With respect to this grading, for $n \in \mathbb N$ the $n$-homogeneous component of $U$ is the
$\natural$-image of the $n$-homogeneous component of $U^+_q$. This homogeneous component is equal to $\mathbb V_n \cap U$.
\medskip

\noindent The following result is a variation on \cite[Theorem~5]{leclerc}.

\begin{proposition}\label{prop:orthog} {\rm (See \cite[Lemma~6.5]{boxq}.)}
 The ideal $J$ from Definition \ref{def:J} and the subalgebra $U$ from Definition \ref{def:Usub} 
are orthogonal complements with respect to the bilinear form $\langle \,,\,\rangle$.
\end{proposition}

\noindent In \cite[Theorem~1.7]{catalan} we applied the map $\natural $ to each element in the Damiani PBW basis for $U^+_q$,
and expressed the image in the standard basis
   for $\mathbb V$.
  We will review this result in Proposition \ref{thm:mainres} below. In order to prepare for Proposition \ref{thm:mainres}, we make some comments.

\begin{definition}
\label{def:CAT1} \rm
Define $\overline x = 1$
and $\overline y = -1$.
A word $u_1u_2\cdots u_n$
in $\mathbb V$ is  said to be {\it Catalan} whenever 
$\overline u_1+
\overline u_2+\cdots +
\overline u_i$ is nonnegative for $1 \leq i \leq n-1$ and zero for $i=n$.
 In this case $n$ is even.
\end{definition}

\begin{example}
\label{ex:CatEX} 
For $0\leq n \leq 3$ we display the Catalan words of length $2n$.
\bigskip

\centerline{
\begin{tabular}[t]{c|c}
   $n$  & {\rm Catalan words of length $2n$} 
   \\
   \hline
 $ 0 $  &  $1$
 \\
 $ 1 $  &  $xy$
 \\
 $ 2 $  &  $xyxy, \quad xxyy$
 \\
 $ 3 $  & 
 $xyxyxy,
 \quad xxyyxy,
 \quad xyxxyy,
 \quad xxyxyy,
 \quad xxxyyy$
   \end{tabular}}
\end{example}

\begin{definition} 
\label{def:CnIntro} {\rm (See \cite[Definition~1.5]{catalan}.)}
\rm For $n\in \mathbb N$ define
\begin{align}
C_n = 
&  \sum u_1u_2\cdots u_{2n} 
\lbrack 1\rbrack_q
\lbrack 1+\overline u_1\rbrack_q
\lbrack 1+\overline u_1+\overline u_2\rbrack_q
\cdots 
\lbrack 1+\overline u_1+\overline u_2+ \cdots +\overline u_{2n}\rbrack_q,
\label{eq:cdefIntro}
\end{align}
where the sum is over all the Catalan words $u_1 u_2 \cdots u_{2n}$
in $\mathbb V$ that have length $2n$. 
We call $C_n$ the {\it $n^{\rm th}$ Catalan element} in $\mathbb V$. Note that $C_n \in \mathbb V_{2n}$.
\end{definition}

\begin{example} 
\label{ex:CnIntro} We have
\begin{align*}
&\qquad \qquad C_0 = 1,
\qquad \qquad
 C_1 = \lbrack 2 \rbrack_q  xy,
\qquad \qquad
 C_2 = \lbrack 2\rbrack^2_q xyxy+ 
\lbrack 3 \rbrack_q \lbrack 2 \rbrack^2_q 
 xxyy,
\\
&
C_3 = 
\lbrack 2 \rbrack^3_q
      xyxyxy + 
       \lbrack 3\rbrack_q
\lbrack 2 \rbrack^3_q
      xxyyxy
     +
     \lbrack 3 \rbrack_q
\lbrack 2 \rbrack^3_q
     xyxxyy
     +
     \lbrack 3 \rbrack^2_q
\lbrack 2 \rbrack^3_q
     xxyxyy
     +
\lbrack 4 \rbrack_q
\lbrack 3 \rbrack^2_q 
\lbrack 2 \rbrack^2_q 
     xxxyyy.
\end{align*}
\end{example}

\begin{proposition} {\rm (See \cite[Theorem~1.7]{catalan}.)}
\label{thm:mainres} 
The map $\natural$ sends
\begin{align}
E_{n\delta+\alpha_0} \mapsto  q^{-2n}(q-q^{-1})^{2n} xC_n,
\qquad  \quad 
E_{n\delta+\alpha_1} \mapsto
 q^{-2n}(q-q^{-1})^{2n} C_ny 
\label{eq:main1}
\end{align}
for $n\geq 0$, and
\begin{align}
E_{n\delta} \mapsto  -q^{-2n}(q-q^{-1})^{2n-1} C_n
\label{eq:main2}
\end{align}
for $n\geq 1$.
\end{proposition}
\noindent We emphasize that in 
{\rm (\ref{eq:main1})}, 
the notations $xC_n$ and $C_ny$ refer to the free product.
\medskip

\noindent We mention three consequences of Proposition \ref{thm:mainres}.

\begin{corollary}
    \label{prop:PBWbasisC}
A PBW basis for $U$ is obtained by the elements
\begin{align*}
\lbrace xC_n \rbrace_{n=0}^\infty, \qquad \quad
\lbrace C_n y \rbrace_{n=0}^\infty, \qquad \quad
\lbrace C_n  \rbrace_{n=1}^\infty
\end{align*}
  in the 
     linear
    order
 \begin{align*}
 x < xC_1< xC_2 < \cdots < C_1 < C_2 < C_3 < \cdots < C_2y < C_1y < y.
     \end{align*}
   \end{corollary}
\begin{proof} By Propositions    \ref{prop:PBWbasis}, \ref{thm:mainres}.
\end{proof}

\begin{corollary} {\rm (See \cite[Corollary~1.8]{catalan}.)}
\label{cor:com}
For $i,j \in \mathbb N$,
\begin{equation}
C_i \star C_j = C_j \star C_i.
\end{equation}
\end{corollary}
\begin{proof}  By Lemma \ref{lem:mc} and Proposition \ref{thm:mainres}.
\end{proof}

\begin{corollary} \label{cor:Qcom}
{\rm (See \cite[Corollary~3.6]{catalan}.)}
For $i,j\in \mathbb N$,
\begin{align}
q^{-1}C_{i+j+1} = \frac{ q (xC_i)\star (C_jy)-q^{-1} (C_j y)\star (x C_i)}
{q-q^{-1}}. \label{eq:CCC}
\end{align}
\end{corollary}
\begin{proof} By Lemma
\ref{lem:dam2} and Proposition \ref{thm:mainres}.
\end{proof}

\noindent We just displayed some relations involving the Catalan elements. 
Additional relations involving the Catalan elements can be found in \cite[Section~3]{catalan}.

 \section{The main result}
 
 \noindent In this section we prove our main result, which is Theorem \ref{thm:Beckmap}.
  Recall the map $\natural $ from below Definition \ref{def:Usub}.

 \begin{theorem} \label{thm:Beckmap}
The map
$\natural$ sends

\begin{align}
E^{\rm Beck}_{n\delta} \mapsto \frac{\lbrack 2n \rbrack_q}{n} q^{-2n} (q-q^{-1})^{2n-1} x C_{n-1} y
\label{eq:xcy}
\end{align}
 for $n\geq 1$. 
\end{theorem}
\noindent We emphasize that in \eqref{eq:xcy} the notation $xC_{n-1}y$ refers to the free product. This notation is illustrated in Example \ref{ex:xCy}.
\medskip

\noindent We will prove  Theorem \ref{thm:Beckmap} after two preliminary lemmas.

\begin{lemma}\label{lem:inU} For $k \in \mathbb N$ we have $x C_k y \in U$.
\end{lemma}
\begin{proof} By Proposition \ref{prop:orthog}, it suffices to show that $xC_ky$
is orthogonal to everything in $J$. By Definition \ref{def:J} and the construction, the vector space $J$ is spanned by the elements of the form
$w_1 J^{\pm} w_2$, where $w_1, w_2$ are words in $\mathbb V$.
Let $w_1$, $w_2$ denote words in $\mathbb V$. We will show that
\begin{align}
\langle xC_k y, w_1 J^+ w_2\rangle = 0, \qquad \qquad \langle xC_k y, w_1 J^- w_2 \rangle = 0.
\label{eq:wJw}
\end{align}
Observe that $x C_k y \in \mathbb V_{2k+2}$ and $J^\pm \in\mathbb V_4$. 
We may assume that $k\geq 1$ and ${\rm length}(w_1)+{\rm length}(w_2)= 2k-2$; otherwise $xC_k y$ and $w_1 J^\pm w_2$ are in different homogeneous components of $\mathbb V$, in which  case
\eqref{eq:wJw} holds.
\noindent We now investigate four cases.
First assume that $w_1$ and $w_2$ are trivial. We have $k=1$. We have $C_1 = \lbrack 2 \rbrack_q xy$ by Example \ref{ex:CnIntro},
and $\langle xxyy, J^\pm\rangle = 0$ by Definition \ref{def:Jpm}. 
By these comments \eqref{eq:wJw} holds.
Next assume that $w_1$ is trivial and $w_2$ is nontrivial. We have $k\geq 2$. By Definition \ref{def:CnIntro}, $C_k$ is a linear combination of the Catalan words in $\mathbb V$ that have length $2k$. For such a word $w$ its first three letters form one of the words
$xxx$ or $xxy$ or $xyx$. By this and the coefficient formula \eqref{eq:cdefIntro},
we obtain
\begin{align*}
C_k = xxx R_1 +( \lbrack 3 \rbrack_q xxy+xyx)R_2,  \qquad \qquad  R_1, R_2 \in \mathbb V_{2k-3}.
\end{align*}
By Definition \ref{def:Jpm} we obtain
\begin{align*}
&\langle xxxx, J^+\rangle = 0, \qquad \quad 
 \langle xxxy, J^+ \rangle = 1, \qquad \quad 
\langle xxyxJ^+ \rangle = -\lbrack 3 \rbrack_q, \\
&\langle xxxx, J^-\rangle = 0, \qquad \quad 
\langle xxxy, J^- \rangle = 0, \qquad \quad 
\langle xxyx, J^- \rangle = 0.
\end{align*}
By this  and Lemma \ref{lem:XY}, 
\begin{align*}
\langle xC_k y , w_1 J^\pm w_2\rangle &= \langle xC_k y, J^\pm w_2\rangle \\
&= \langle xxxxR_1y + (\lbrack 3 \rbrack_q xxxy + xxyx)R_2y, J^\pm w_2 \rangle \\
&= \langle xxxxR_1y, J^\pm w_2 \rangle +\langle (\lbrack 3 \rbrack_q xxxy + xxyx)R_2y, J^\pm w_2 \rangle \\
&=
\langle xxxx, J^\pm \rangle \langle R_1y, w_2\rangle + \langle \lbrack 3 \rbrack_q xxxy+xxyx, J^\pm\rangle \langle R_2y, w_2\rangle \\
&=0.
\end{align*}
We have established \eqref{eq:wJw}  for this case.
Next assume that $w_1$ is nontrivial and $w_2$ is trivial. We have $k\geq 2$. Adjusting the argument of the previous case, we obtain
\begin{align*}
C_k =  L_1 yyy +L_2( \lbrack 3 \rbrack_q xyy+yxy),  \qquad \qquad  L_1, L_2 \in \mathbb V_{2k-3}.
\end{align*}
By Definition \ref{def:Jpm} we obtain
\begin{align*}
&\langle yyyy, J^+\rangle = 0, \qquad \quad 
 \langle xyyy, J^+ \rangle = 0, \qquad \quad 
\langle yxyy, J^+ \rangle = 0, \\
&\langle yyyy, J^-\rangle = 0, \qquad \quad 
\langle xyyy, J^- \rangle = -1, \qquad \quad 
\langle yxyy, J^- \rangle = \lbrack 3 \rbrack_q.
\end{align*}
By this  and Lemma \ref{lem:XY}, 
\begin{align*}
\langle xC_k y , w_1 J^\pm w_2\rangle &= \langle xC_k y, w_1J^\pm \rangle \\
& =  \langle xL_1yyyy  + xL_2( \lbrack 3 \rbrack_q xyyy+yxyy), w_1J^\pm \rangle \\
& =  \langle xL_1yyyy, w_1J^\pm \rangle +
 \langle xL_2( \lbrack 3 \rbrack_q xyyy+yxyy), w_1J^\pm \rangle \\
&=
\langle xL_1, w_1 \rangle \langle yyyy, J^\pm\rangle + \langle xL_2, w_1\rangle \langle   \lbrack 3 \rbrack_q xyyy+yxyy, J^\pm \rangle \\
&=0.
\end{align*}
We have established \eqref{eq:wJw}  for this case.
Next assume that each of $w_1$, $w_2$ is nontrivial. There exist letters $a$, $b$ and words $w'_1$, $w'_2$
such that $w_1 = a w'_1$ and $w_2 = w'_2 b$. We have
$C_k \in U$ and $w'_1 J^\pm w'_2\in J$ and $\langle U, J\rangle = 0$, so
\begin{align*}
 \langle C_k, w'_1 J^\pm w'_2 \rangle =0.
\end{align*}
By this and Lemma \ref{lem:XY},
\begin{align*}
\langle xC_k y , w_1 J^\pm w_2\rangle = 
\langle xC_k y , aw'_1 J^\pm w'_2b\rangle =
 \langle x, a \rangle \langle C_k, w'_1 J^\pm w'_2 \rangle \langle y,b\rangle 
=0.
\end{align*}
We have established \eqref{eq:wJw} for this case.
The condition \eqref{eq:wJw} holds in all four cases, and the result follows.
\end{proof}

\begin{lemma} \label{lem:n1n2} For $k \in \mathbb N$ we have
\begin{align}
xC_{k+1} &= \frac{x \star (x C_k y) -(x C_k y) \star x}{q-q^{-1}},    \label{eq:n1}
\\
C_{k+1} y &= \frac{(x C_k y) \star y - y \star (xC_k y)}{q-q^{-1}}.  \label{eq:n2}
\end{align}
\end{lemma}
\begin{proof} We first verify \eqref{eq:n1}. Consider the left-hand side of 
 \eqref{eq:n1}.
Setting $i=0$ and $j=k$ in \eqref{eq:CCC}, we obtain
\begin{align*}
C_{k+1} = \frac{q^2 x \star (C_k y) - (C_k y) \star x}{q-q^{-1}}.
\end{align*}
Therefore
\begin{align}
xC_{k+1} = x\frac{q^2 x \star (C_k y) - (C_k y) \star x}{q-q^{-1}}.
\label{eq:xC}
\end{align}
Now consider the right-hand side of
\eqref{eq:n1}. 
Using \eqref{eq:uvcirc} we obtain $x \star (x C_k y) = x\bigl( x C_k y +q^2 x \star (C_k y)\bigr)$ and
$(x C_k y) \star x = x \bigl( (C_k y)\star x + x C_k y\bigr)$. By these comments, the right-hand side of \eqref{eq:n1} is equal to the right-hand side of
\eqref{eq:xC}. We have verified \eqref{eq:n1}. The equation \eqref{eq:n2} is verified in a similar way.
\end{proof}
\noindent {\it Proof of Theorem \ref{thm:Beckmap}}. Define $\mathcal C_n \in U$ such that the map $\natural$ sends
\begin{align}
E^{\rm Beck}_{n\delta} \mapsto \frac{\lbrack 2n \rbrack_q}{n} q^{-2n} (q-q^{-1})^{2n-1} \mathcal C_n.
\label{eq:calC}
\end{align}
We show that $\mathcal C_n = xC_{n-1}y$. For \eqref{eq:con2} and \eqref{eq:con1}, apply $\natural $ to each side 
and evaluate the result using \eqref{eq:main1}, \eqref{eq:calC}. The result is
\begin{align}
x C_{k+\ell}&=\frac{(x C_\ell)\star \mathcal C_k - \mathcal C_k \star (x C_\ell)}{q-q^{-1}}, \qquad 
C_{k+\ell}y =\frac{\mathcal C_k \star (C_\ell y) - (C_\ell y) \star \mathcal C_k}{q-q^{-1}}. 
\label{lem:gen}
\end{align}
Setting $\ell=0$ and $k=n$ in \eqref{lem:gen}, we obtain
 \begin{align}
x C_{n}&=\frac{x\star \mathcal C_n - \mathcal C_n \star x}{q-q^{-1}}, \qquad \qquad
C_{n}y =\frac{\mathcal C_n \star y - y \star \mathcal C_n}{q-q^{-1}}. 
\label{lem:2gen}
\end{align}
Setting $k=n-1$ in Lemma \ref{lem:n1n2}, we obtain
 \begin{align}
x C_{n}&=\frac{x\star (xC_{n-1}y) - (xC_{n-1}y)\star x}{q-q^{-1}}, \qquad 
C_{n}y =\frac{(xC_{n-1}y) \star y - y \star (xC_{n-1}y)}{q-q^{-1}}. 
\label{lem:2gena}
\end{align}
\noindent Consider the element $\mathcal C_n-xC_{n-1}y$. 
By Lemma \ref{lem:inU} and the construction, $\mathcal C_n - xC_{n-1}y \in U$.
By the equations on the left   in \eqref{lem:2gen} and \eqref{lem:2gena}, 
$\mathcal C_n - xC_{n-1}y$ commutes with $x$ with respect to $\star $.
 By the equations on the right  in \eqref{lem:2gen} and \eqref{lem:2gena}, 
$\mathcal C_n - xC_{n-1}y$ commutes with $y$ with respect to $\star$. By these comments $\mathcal C_n-xC_{n-1}y$ is contained in the center of $U$. The algebra $U$ is isomorphic to $U^+_q$, so
by \cite[Lemma~6.1]{altCE} the center of $U$ is equal to $\mathbb F 1$. Therefore, there exists $\alpha_n \in \mathbb F$ such that $\mathcal C_n-xC_{n-1}y=\alpha_n 1$. We show that $\alpha_n =0$.
By Lemma \ref{lem:BeckG} and \eqref{eq:calC} along with our comments above Proposition \ref{prop:orthog},
 we obtain $\mathcal C_n \in \mathbb V_{2n}$.
We have $C_{n-1} \in \mathbb V_{2n-2}$ so $xC_{n-1} y \in \mathbb V_{2n}$. By these comments
 $\alpha_n 1 = \mathcal C_n-xC_{n-1}y \in \mathbb V_{2n}$.
However $\alpha_n 1 \in \mathbb V_0$ and $\mathbb V_0 \cap \mathbb V_{2n}=0$ since $n\geq 1$, so $\alpha_n 1=0$. Therefore $\alpha_n = 0$.
 We have shown that $\mathcal C_n = x C_{n-1}y$, as desired. \hfill $\Box$
  \\
\section{Some consequences of the main result}

\noindent In this section we give some consequences of our main result Theorem \ref{thm:Beckmap}.

\begin{corollary} \label{cor:xCyC} The following holds in the $q$-shuffle algebra $\mathbb V$:
\begin{align}
&{\rm exp}  \Biggl( \sum_{k=1}^\infty \frac{ \lbrack 2 k \rbrack_q }{k} x C_{k-1} y t^k \Biggr) = 1 + \sum_{k=1}^\infty C_k t^k.     \label{eq:expC}
\end{align}
\end{corollary}
\noindent We emphasize that in \eqref{eq:expC} the exponential function is with respect to the $q$-shuffle product, and the notation $xC_{k-1}y$ refers to the free product.
\begin{proof} Apply the map $\natural $ to each side of \eqref{eq:expEE}, and evaluate the result using \eqref{eq:main2}, \eqref{eq:xcy}. This yields \eqref{eq:expC}
after a change of variables in which $t$ is replaced by $q^2(q-q^{-1})^{-2}t$.
\end{proof}

 \noindent Using \eqref{eq:expC}  the elements $\lbrace xC_ny \rbrace_{n=0}^\infty$ and the elements $\lbrace C_n \rbrace_{n=1}^\infty$ can be
 recursively obtained from each other. This is illustrated in Example  \ref{ex:xCyC}.
   
\begin{corollary} \label{lem:Poly} 
For $n\geq 1$ the following hold in the $q$-shuffle algebra $\mathbb V$:
\begin{enumerate}
\item[\rm (i)]  $C_n $ is a homogeneous polynomial in $xC_0y, xC_1y, \ldots, xC_{n-1}y$ that has total degree $n$, where we view
$xC_{k-1}y$ as having degree $k$ for $1 \leq k \leq n$;
\item[\rm (ii)] $xC_{n-1}y$ is a homogeneous polynomial in $C_1, C_2,\ldots, C_n$ that has total degree $n$, where we view
$C_k$ as having degree $k$ for $1 \leq k \leq n$.
\end{enumerate}
\end{corollary}
\noindent We emphasize that the above homogeneous polynomials are with respect to the $q$-shuffle product.
  \begin{proof}  Use \eqref{eq:expC} and induction on $n$. 
   \end{proof}

\begin{corollary}
    \label{prop:PBWbasisXCY}
A PBW basis for $U$ is obtained by the elements
\begin{align*}
\lbrace xC_n \rbrace_{n \in \mathbb N},\qquad \quad
\lbrace C_n y \rbrace_{n\in \mathbb N}, \qquad \quad
\lbrace xC_ny  \rbrace_{n\in \mathbb N}
\end{align*}
  in the
     linear
    order
\begin{align*}
x < xC_1< xC_2 < \cdots < xy < xC_1y < xC_2y < \cdots < C_2y < C_1y < y.
\end{align*}    
   \end{corollary}
   \begin{proof}  Apply the map $\natural $ to everything in   Proposition \ref{prop:Beckpbw},         and evaluate the result using \eqref{eq:main1}, \eqref{eq:xcy}. 
   \end{proof}
   
\begin{corollary}
\label{cor:xcyCom} 
The elements $\lbrace xC_ny \rbrace_{n \in \mathbb N}$ mutually commute with respect to the $q$-shuffle product.
\end{corollary}
\begin{proof} By Lemma \ref{lem:Bmc} and \eqref{eq:xcy}. 
\end{proof}

\begin{corollary} For $k, \ell \in \mathbb N$,
\begin{align}
x C_{k+\ell+1}&=\frac{(x C_\ell)\star (x C_k y) - (x C_k y) \star (x C_\ell)}{q-q^{-1}}, \label{eq:xCy1}
\\
C_{k+\ell+1}y &=\frac{(x C_k y)\star (C_\ell y) - (C_\ell y) \star (x C_k y)}{q-q^{-1}}. \label{eq:xCy2}
\end{align}
\end{corollary}
 \begin{proof}  For   \eqref{eq:con2} and  \eqref{eq:con1},   apply the map $\natural $ to each side,   and evaluate the results using \eqref{eq:main2}, \eqref{eq:xcy}. This yields
  \eqref{eq:xCy1}, \eqref{eq:xCy2} after a change of variables in which $k$ is replaced by $k+1$.
   \end{proof}

\section{The alternating words and the Catalan elements}

\noindent In \cite{alternating} we introduced the alternating PBW basis for $U^+_q$. In this section we discuss how certain elements of  this PBW basis are
related to  $\lbrace C_n  \rbrace_{n \in \mathbb N}$ and  $\lbrace xC_n y \rbrace_{n \in \mathbb N}$.
\medskip

\noindent
We recall the alternating words in $\mathbb V$.

\begin{definition}\rm {\rm (See \cite[Definition~5.1]{alternating}.)}
A word
$u_1u_2\cdots u_n$ in $\mathbb V$ is called {\it alternating}
whenever $n\geq 1$ and 
$u_{i-1} \not=u_i$ for $2 \leq i \leq n$.
Thus an alternating word has the form $\cdots xyxy\cdots$.
\end{definition}

\begin{definition} {\rm (See \cite[Definition~5.2]{alternating}.)}
\label{def:WWGG}
\rm We name the alternating words as follows:
\begin{align*}
&W_0 = x, \qquad W_{-1} = xyx, \qquad W_{-2} = xyxyx, \qquad \ldots
\\
&W_1 = y, \qquad W_{2} = yxy, \qquad W_{3} = yxyxy, \qquad \ldots 
\\
&G_{1} = yx, \qquad G_{2} = yxyx,  \qquad G_3 = yxyxyx, \qquad \ldots 
\\
&\tilde G_{1} = xy, \qquad \tilde G_{2} =
xyxy,\qquad \tilde G_3 = xyxyxy, \qquad \ldots
\end{align*}
For notational convenience, define $G_0=1$ and $\tilde G_0=1$.
\end{definition}

\noindent In \cite[Propositions~5.7,~5.10,~5.11,~6.3,~8.1]{alternating} we displayed many relations involving the alternating words. These relations show how the alternating words are related to each other, with respect to the $q$-shuffle product. Using these relations and referring to the $q$-shuffle product,
in \cite[Theorem~10.1]{alternating} we recursively obtained each alternating
word as a polynomial in $x, y$. This result has the following consequence.

\begin{lemma} {\rm (See \cite[Theorem~8.3]{alternating}.)} Each alternating word in $\mathbb V$ is contained in $U$.
\end{lemma}
\noindent Next we describe the alternating PBW basis for $U$.

\begin{proposition} {\rm (See \cite[Theorem~10.1]{alternating}.)}
A PBW basis for $ U$ is obtained by the elements
\begin{align*}
\lbrace W_{-i} \rbrace_{i \in \mathbb N}, \qquad 
\lbrace \tilde G_{j+1} \rbrace_{j\in \mathbb N}, \qquad  
\lbrace W_{k+1} \rbrace_{k\in \mathbb N}
\end{align*}
in any linear order $<$ that satisfies 
\begin{align*}
W_{-i} <  
\tilde G_{j+1} < W_{k+1}
\qquad \quad i,j,k \in \mathbb N.
\end{align*}
\end{proposition}
\noindent For the rest of this section, we focus on the elements $\lbrace \tilde G_n\rbrace_{n \in \mathbb N}$. In \cite[Section~11]{alternating} we described how the elements
 $\lbrace C_n \rbrace_{n \in \mathbb N}$ are related to the elements
$\lbrace \tilde G_n \rbrace_{n \in \mathbb N}$. We will review this description, and then describe how the elements $\lbrace xC_ny \rbrace_{n \in \mathbb N}$ are related to the elements
$\lbrace \tilde G_n \rbrace_{n \in \mathbb N}$. Our main result on this topic is Proposition \ref{thm:GxCX}  below.
\begin{lemma} {\rm (See \cite[Proposition~5.10]{alternating}.)} The elements $\lbrace \tilde G_{n} \rbrace_{n \in \mathbb N}$ mutually commute with respect to the
$q$-shuffle product.
\end{lemma}

\begin{definition}\rm We define some generating functions in the indeterminate $t$:
\begin{align*}
C(t) &= \sum_{n \in \mathbb N} C_n t^n,
\qquad \qquad
\tilde G(t) = \sum_{n \in \mathbb N} \tilde G_n t^n.
\end{align*}
\end{definition}

 \begin{proposition} \label{lem:GGC} {\rm (See       \cite[Lemma~9.12~and~Proposition~11.8]{alternating}.)}We have
 \begin{align}
 \label{eq:GGC}
 \tilde G(qt) \star C(-t) \star  \tilde G(q^{-1} t) =1.
 \end{align}
 \end{proposition}
 
 \begin{corollary} \label{lem:recGC} {\rm (See \cite[Theorem~11.14]{alternating}.)} For $n \in \mathbb N$,
 \begin{align}
 \label{eq:CG}
 0 &= \sum_{i=0}^n (-1)^i \lbrack 2n-i\rbrack_q C_i \star \tilde G_{n-i}.
 \end{align}
 \end{corollary}
 
 \noindent The following result is obtained by rearranging the terms in \eqref{eq:CG}.
 \begin{corollary}\label{cor:CGGC} For $n\geq 1$,
 \begin{align*}
  C_n &= \frac{-1}{\lbrack n \rbrack_q} \sum_{i=0}^{n-1} (-1)^{n-i} \lbrack 2n-i \rbrack_q C_i \star \tilde G_{n-i},
  \\
 \tilde G_n &= \frac{-1}{\lbrack 2n \rbrack_q} \sum_{i=1}^n (-1)^{i} \lbrack 2n-i \rbrack_q C_i \star \tilde G_{n-i}.
 \end{align*}
 \end{corollary}
 \noindent Using Corollary \ref{cor:CGGC}  the elements $\lbrace C_n \rbrace_{n=1}^\infty$ and the elements $\lbrace \tilde G_n \rbrace_{n=1}^\infty$ can be
 recursively obtained from each other. This is illustrated in Example  \ref{lem:CG}.
 
 \begin{corollary}
\label{prop:CGpoly} {\rm (See \cite[Corollary~11.11]{alternating}.)}
For $n\geq 1$ the following hold in the $q$-shuffle algebra $\mathbb V$.
\begin{enumerate}
\item[\rm (i)] $C_n$ is a homogeneous polynomial in 
$\tilde G_1, \tilde G_2,\ldots, \tilde G_n$ that has total degree $n$,
where we view $\tilde G_k$ as having degree $k$ for $1 \leq k \leq n$;
\item[\rm (ii)] $\tilde G_n$ is a homogeneous polynomial in 
$C_1, C_2,\ldots, C_n$ that has total degree $n$, where we view
$C_k$ as having degree $k$ for $1 \leq k \leq n$.
\end{enumerate}
\end{corollary}
\noindent We emphasize that the above homogeneous polynomials are with respect to the $q$-shuffle product.
 \medskip
 
\noindent Next we describe how the elements $\lbrace xC_ny \rbrace_{n \in \mathbb N}$ are related to the elements $\lbrace \tilde G_n \rbrace_{n \in \mathbb N}$.
Our description gives a variation on a formula in \cite[Proposition~5.27]{basFMA} involving a certain generating function $\mathcal G_+(u)$ that corresponds to 
$\tilde G(t)$.

\begin{proposition} \label{thm:GxCX} The following holds in the $q$-shuffle algebra $\mathbb V$:
\begin{align}
\label{eq:E1b}
&{\rm exp}  \Biggl(-\sum_{k=1}^\infty \frac{ (-1)^k \lbrack  k \rbrack_q }{k} x C_{k-1} y t^k \Biggr) = 1+ \sum_{k=1}^\infty \tilde G_k t^k.
\end{align}
\end{proposition}
\noindent We emphasize that in \eqref{eq:E1b} the exponential function  is with respect to the $q$-shuffle product, and the notation $xC_{k-1}y$ refers to the free product.
\begin{proof} Recall that $\tilde G_0=1$. Define the generating function 
$\tilde g(t)=\sum_{k=1}^\infty \tilde G_k t^k$ and note that $\tilde G(t)= 1+ \tilde g(t)$.
 We will be discussing the natural logarithm ${\rm ln}$ with respect to $\star$.
Define $h(t) = {\rm ln} \bigl(\tilde G(t)\bigr)$, and  note that
\begin{align*}
h(t) = {\rm ln} \bigl(1+ \tilde g(t)\bigr) = \tilde g(t) - \frac{\tilde g(t) \star \tilde g(t)}{2} + \frac{\tilde g(t)\star \tilde g(t) \star \tilde g(t)}{3} - \cdots
\end{align*}
We have ${\rm exp}\,h(t) = \tilde G(t)$. To establish \eqref{eq:E1b}, it suffices to show that
\begin{align}
h(t) = - \sum_{k=1}^\infty \frac{(-1)^k \lbrack k \rbrack_q}{k}  x C_{k-1}y t^k.
\label{eq:hk}
\end{align}
By construction, $h(t)$ has constant coefficient 0. 
Write $h(t)=\sum_{k=1}^\infty h_k t^k$ with $h_k \in U$ for $k\geq 1$.
Applying $\rm ln$ to each side of \eqref{eq:GGC}, we obtain
\begin{align}
\label{eq:three}
{\rm ln} \bigl(\tilde G(qt)\bigr) + {\rm ln}  \bigl(C(-t)\bigr) + {\rm ln}\bigl(\tilde G(q^{-1}t)\bigr) = 0.
\end{align}
By construction
\begin{align}
&{\rm ln} \bigl(\tilde G(qt)\bigr) = h(qt) = \sum_{k=1}^\infty h_k q^k t^k,
\label{eq:hq} \\
&{\rm ln} \bigl(\tilde G(q^{-1}t)\bigr) = h(q^{-1}t) = \sum_{k=1}^\infty h_k q^{-k} t^k.
\end{align}
By Corollary \ref{cor:xCyC},
\begin{align}
{\rm ln}  \bigl(C(-t)\bigr) = \sum_{k=1}^\infty \frac{(-1)^k \lbrack 2k\rbrack_q}{k} x C_{k-1} y t^k.
\label{eq:Ct}
\end{align}
Evaluating \eqref{eq:three} using \eqref{eq:hq}--\eqref{eq:Ct}, we obtain
\begin{align}
\label{eq:hsolve}
h_k (q^k +q^{-k})+ \frac{(-1)^k \lbrack 2k \rbrack_q}{k} x C_{k-1} y = 0, \qquad \qquad k\geq 1.
\end{align}
By \eqref{eq:hsolve} and $\lbrack 2k \rbrack_q = \lbrack k \rbrack_q (q^k+q^{-k})$, 
\begin{align*}
h_k = - \frac{(-1)^k \lbrack k \rbrack_q}{k} x C_{k-1} y, \qquad \qquad k\geq 1.
\end{align*}
This implies \eqref{eq:hk},
and the result follows.
\end{proof}

 \noindent Using \eqref{eq:E1b} the elements $\lbrace xC_ny \rbrace_{n=0}^\infty$ and the elements $\lbrace \tilde G_n \rbrace_{n=1}^\infty$ can be
 recursively obtained from each other. This is illustrated in Example  \ref{ex:xCyG}.
 
  \begin{corollary}
\label{prop:xCyGpoly}
For $n\geq 1$ the following hold in the $q$-shuffle algebra $\mathbb V$.
\begin{enumerate}
\item[\rm (i)] $xC_{n-1}y$ is a homogeneous polynomial in 
$\tilde G_1, \tilde G_2,\ldots, \tilde G_n$ that has total degree $n$,
where we view $\tilde G_k$ as having degree $k$ for $1 \leq k \leq n$;
\item[\rm (ii)] $\tilde G_n$ is a homogeneous polynomial in 
$xC_0y, xC_1y, \ldots, xC_{n-1}y$ that has total degree $n$, where we view
$xC_{k-1}y$ as having degree $k$ for $1 \leq k \leq n$.
\end{enumerate}
\end{corollary}
\noindent We emphasize that the above homogeneous polynomials are with respect to the $q$-shuffle product.
\begin{proof} Use \eqref{eq:E1b} and induction on $n$.
\end{proof}

\begin{corollary} 
\label{cor:CCG}
The following {\rm (i)--(iii)} coincide:
\begin{enumerate}
\item[\rm (i)] the subalgebra of the $q$-shuffle algebra $\mathbb V$
generated by $\lbrace C_n \rbrace_{n=1}^\infty$;
\item[\rm (ii)] the subalgebra of the $q$-shuffle algebra $\mathbb V$
generated by $\lbrace xC_ny \rbrace_{n=0}^\infty$;
\item[\rm (iii)] the subalgebra of the $q$-shuffle algebra $\mathbb V$
generated by $\lbrace \tilde G_n \rbrace_{n=1}^\infty$.
\end{enumerate}
\end{corollary}
\begin{proof} By Corollaries  \ref{lem:Poly},  \ref{prop:CGpoly}, \ref{prop:xCyGpoly}.
\end{proof}

\noindent For more information about the alternating words and related topics, see
\cite{basFMA, BK05, basnc, altCE, compactUqp, factorUq
}.

 \section{Acknowledgement}
 The author thanks Pascal Baseilhac for many discussions about $U^+_q$ and its PBW bases.

\section{Appendix A}
Recall the Catalan elements $\lbrace C_n\rbrace_{n\in \mathbb N}$ from Definition \ref{def:CnIntro}, and the alternating elements $\lbrace \tilde G_n \rbrace_{n \in \mathbb N}$ from
Definition \ref{def:WWGG}. Recall that $C_0=1$ and $\tilde G_0=1$.
In this appendix we compare   $\lbrace x C_n y \rbrace_{n=0}^3$  and  $\lbrace C_n \rbrace_{n=1}^4$  and    $\lbrace \tilde G_n\rbrace_{n=1}^4$.
\medskip

\noindent 
First  we display $xC_ny$ for $0 \leq n \leq 3$.
 \begin{example}\label{ex:xCy} \rm We have
 \begin{align*}
 &xC_0y = xy, \qquad
 x C_1 y = \lbrack 2 \rbrack_q xxyy, \qquad
 xC_2 y = \lbrack 2 \rbrack^2_q xxyxyy+ \lbrack 3 \rbrack_q \lbrack 2 \rbrack^2_q xxxyyy,
 \\
 & xC_3 y = 
 \lbrack 2 \rbrack^3_q xxyxyxyy+
 \lbrack 3 \rbrack_q \lbrack 2 \rbrack^3_q xxxyyxyy+
 \lbrack 3 \rbrack_q \lbrack 2 \rbrack^3_q xxyxxyyy
 \\&\qquad \qquad +
 \lbrack 3 \rbrack^2_q \lbrack 2 \rbrack^3_q xxxyxyyy+
 \lbrack 4 \rbrack_q \lbrack 3 \rbrack^2_q \lbrack 2 \rbrack^2_q xxxxyyyy.
 \end{align*}
 \end{example}
 \noindent Next we compare $\lbrace x C_ny \rbrace_{n=0}^3$ and $\lbrace C_n \rbrace_{n=1}^4$.
 
 \begin{example}\label{ex:xCyC} \rm We have
 \begin{align*}
x C_0 y &= \frac{C_1} {\lbrack 2 \rbrack_q},
\qquad \quad 
x C_1 y = \frac{ 2 C_2 - C_1 \star C_1}{\lbrack 4 \rbrack_q},
\\
x C_2 y &= \frac{ 3 C_3 - 3 C_2 \star C_1 +C_1 \star C_1 \star C_1}{\lbrack 6 \rbrack_q},
\\
x C_3 y &= \frac{ 4C_4 -4 C_3 \star C_1-2 C_2 \star C_2 + 4 C_2 \star C_1 \star C_1 -C_1 \star C_1 \star C_1 \star C_1}{\lbrack 8 \rbrack_q}.
\end{align*}
\noindent Moreover
 \begin{align*} 
 C_1 &= \lbrack 2 \rbrack_q xC_0 y,
 \qquad \quad 
 C_2 =\frac{ \lbrack 4 \rbrack_q x C_1 y + \lbrack 2 \rbrack^2_q (x C_0 y) \star (x C_0 y)} {2},
 \\
 C_3 &= \frac{ 2\lbrack 6 \rbrack_q x C_2 y + 3 \lbrack 2 \rbrack_q \lbrack 4 \rbrack_q (x C_1 y)\star (x C_0 y) + \lbrack 2 \rbrack^3_q (xC_0y) \star (xC_0y)\star (xC_0 y)}{6},
 \\
 C_4 &= \frac{6 \lbrack 8 \rbrack_q x C_3 y + 8 \lbrack 6\rbrack_q \lbrack 2 \rbrack_q (xC_2 y) \star (xC_0y) + 3 \lbrack 4 \rbrack^2_q (xC_1 y)\star (xC_1 y)}{24}
 \\
 &+ \frac{ 6 \lbrack 4 \rbrack_q \lbrack 2 \rbrack^2_q (xC_1 y)\star (xC_0 y)\star (xC_0 y) + \lbrack 2 \rbrack^4_q (xC_0 y)\star (x C_0 y) \star (x C_0 y) \star (x C_0 y)}{24}.
 \end{align*}

\end{example}

 \noindent Next we compare $\lbrace C_n \rbrace_{n=1}^4$ and $\lbrace \tilde G_n \rbrace_{n=1}^4$.
 \begin{example} \label{lem:CG} \rm
 We have 
  \begin{align*}
 &C_1 = \lbrack 2 \rbrack_q \tilde G_1,
 \qquad \quad 
 C_2 = \frac{\lbrack 2 \rbrack_q \lbrack 3 \rbrack_q \tilde G_1\star \tilde G_1 - \lbrack 4 \rbrack_q \tilde G_2}{\lbrack 2 \rbrack_q},
 \\
 &C_3 = \frac{\lbrack 2 \rbrack_q \lbrack 6 \rbrack_q  \tilde G_3 -(\lbrack 4 \rbrack^2_q + \lbrack 2 \rbrack^2_q \lbrack 5 \rbrack_q) \tilde G_2\star  \tilde G_1 + \lbrack 2 \rbrack_q \lbrack 3 \rbrack_q \lbrack 4 \rbrack_q \tilde G_1\star \tilde G_1 \star \tilde G_1}{\lbrack 2 \rbrack_q \lbrack 3 \rbrack_q},
 \end{align*}
  $C_4= \lbrack 2 \rbrack^{-1}_q \lbrack 3 \rbrack^{-1}_q \lbrack 4 \rbrack^{-1}_q$ times a weighted sum with the following terms and coefficients:
 \bigskip
 
 \centerline{
 \begin{tabular}{c|c}
{\rm term} &  {\rm coefficient}
\\
\hline
$\tilde G_4$ & $-\lbrack 2 \rbrack_q\lbrack 3 \rbrack_q \lbrack 8 \rbrack_q $ 
\\
$\tilde G_3 \star \tilde G_1$ & $\lbrack 2 \rbrack^2_q \lbrack 3\rbrack_q  \lbrack 7 \rbrack_q+ \lbrack 2 \rbrack_q \lbrack 5\rbrack_q \lbrack 6\rbrack_q $
\\
$\tilde G_2 \star \tilde G_2 $ & $\lbrack 3 \rbrack_q \lbrack 4 \rbrack_q\lbrack 6 \rbrack_q $
\\
$\tilde G_2 \star \tilde G_1 \star \tilde G_1 $ & $-\lbrack 2 \rbrack_q \lbrack 3 \rbrack^2_q  \lbrack 6 \rbrack_q-\lbrack 2 \rbrack^2_q \lbrack 5 \rbrack^2_q - \lbrack 4 \rbrack^2_q \lbrack 5 \rbrack_q$
\\
$\tilde G_1 \star \tilde G_1 \star \tilde G_1 \star \tilde G_1$ & $\lbrack 2 \rbrack_q \lbrack 3 \rbrack_q \lbrack 4 \rbrack_q\lbrack 5 \rbrack_q$
\end{tabular}
}
\bigskip

\noindent Moreover
 \begin{align*}
 & \tilde G_1 = \frac{C_1}{\lbrack 2 \rbrack_q},
 \qquad \quad 
 \tilde G_2 = \frac{\lbrack 3 \rbrack_q C_1\star C_1 - \lbrack 2 \rbrack^2_q C_2}{\lbrack 2 \rbrack_q \lbrack 4 \rbrack_q},
 \\
 &\tilde G_3 = \frac{\lbrack 2 \rbrack_q \lbrack 3 \rbrack_q \lbrack 4 \rbrack_q C_3 -(\lbrack 4 \rbrack^2_q + \lbrack 2 \rbrack^2_q \lbrack 5 \rbrack_q) C_2\star  C_1 + \lbrack 3 \rbrack_q \lbrack 5 \rbrack_q C_1\star C_1 \star C_1}{\lbrack 2 \rbrack_q \lbrack 4 \rbrack_q \lbrack 6 \rbrack_q},
 \end{align*}
  $\tilde G_4= \lbrack 2 \rbrack^{-1}_q \lbrack 4 \rbrack^{-1}_q \lbrack 6 \rbrack^{-1}_q \lbrack 8 \rbrack^{-1}_q$ times a weighted sum with the following terms and coefficients:
 \bigskip
 
 \centerline{
 \begin{tabular}{c|c}
{\rm term} &  {\rm coefficient}
\\
\hline
$C_4$ & $-\lbrack 2 \rbrack_q\lbrack 4 \rbrack^2_q \lbrack 6 \rbrack_q $ 
\\
$C_3 \star C_1$ & $\lbrack 2 \rbrack_q \lbrack 3\rbrack_q \lbrack 4\rbrack_q  \lbrack 7 \rbrack_q+ \lbrack 4 \rbrack_q \lbrack 5\rbrack_q \lbrack 6\rbrack_q $
\\
$C_2 \star C_2 $ & $\lbrack 2 \rbrack^2_q \lbrack 6 \rbrack^2_q $
\\
$C_2 \star C_1 \star C_1 $ & $-\lbrack 2 \rbrack^2_q \lbrack 5 \rbrack_q  \lbrack 7 \rbrack_q-\lbrack 3 \rbrack_q \lbrack 6 \rbrack^2_q - \lbrack 4 \rbrack^2_q \lbrack 7 \rbrack_q$
\\
$C_1 \star C_1 \star C_1 \star C_1$ & $\lbrack 3 \rbrack_q \lbrack 5 \rbrack_q \lbrack 7 \rbrack_q$
\end{tabular}
}
\bigskip

 \end{example}
 
 \noindent Next we compare $\lbrace x C_n y \rbrace_{n=0}^3$ and
 $\lbrace \tilde G_n\rbrace_{n=1}^4$.
\begin{example}\label{ex:xCyG} \rm We have
\begin{align*}
&xC_0y=\tilde G_1, \qquad \quad  xC_1y = \frac{\tilde G_1 \star \tilde G_1 - 2 \tilde G_2}{\lbrack 2 \rbrack_q}, \\
&xC_2 y = \frac{ \tilde G_1 \star \tilde G_1 \star \tilde G_1-3 \tilde G_1 \star \tilde G_2 + 3 \tilde G_3}{ \lbrack 3 \rbrack_q},
\\
& xC_3 y = \frac{\tilde G_1 \star \tilde G_1 \star \tilde G_1 \star \tilde G_1-4 \tilde G_1 \star \tilde G_1 \star \tilde G_2
 + 2 \tilde G_2 \star \tilde G_2+ 4 \tilde G_1 \star \tilde G_3-4 \tilde G_4}{\lbrack 4 \rbrack_q}.
\end{align*}
\noindent Moreover 
\begin{align*}
& \tilde G_1 = xC_0y, \qquad \quad \tilde G_2 = \frac{(xC_0y)\star (xC_0y) -\lbrack 2 \rbrack_q xC_1y}{2},
\\
& \tilde G_3 = \frac{ (xC_0y)\star (xC_0 y) \star (xC_0y)-3\lbrack 2 \rbrack_q (xC_0y)\star (xC_1y)+2 \lbrack 3 \rbrack_q xC_2 y}{6},
\\
& \tilde G_4 = \frac{(xC_0 y) \star (xC_0 y) \star (xC_0 y) \star (xC_0 y)-6\lbrack 2 \rbrack_q (xC_0y)\star (xC_0 y) \star (xC_1 y)}{24}
\\
&\qquad +\frac{ 3 \lbrack 2\rbrack^2_q (xC_1y) \star (xC_1 y) +8 \lbrack 3 \rbrack_q (xC_0 y)\star (xC_2 y)-6 \lbrack 4 \rbrack_q xC_3 y }{24}.
\end{align*}
\end{example}



%

\bigskip

\noindent Paul Terwilliger \hfil\break
\noindent Department of Mathematics \hfil\break
\noindent University of Wisconsin \hfil\break
\noindent 480 Lincoln Drive \hfil\break
\noindent Madison, WI 53706-1388 USA \hfil\break
\noindent email: {\tt terwilli@math.wisc.edu }\hfil\break

\end{document}